\documentclass{amsart}
\usepackage{graphicx}
\usepackage{latexsym}
\usepackage{color}
\usepackage{amsfonts}
\usepackage[all]{xy}
\usepackage{amssymb, mathrsfs, amsfonts, amsmath}
\usepackage{amsbsy}
\usepackage{amsfonts}
\newtheorem{corollary}{Corollary}

\newtheorem{proposition}{Proposition}
\newtheorem{remark}{Remark}
\newtheorem{theorem}{Theorem}
\newtheorem{example}{Example}
\numberwithin{equation}{section}

\newcommand{\be}{\begin{equation}}
	\newcommand{\ee}{\end{equation}}
\newcommand{\ben}{\begin{enumerate}}
	\newcommand{\een}{\end{enumerate}}
\newcommand{\beq}{\begin{eqnarray}}
	\newcommand{\eeq}{\end{eqnarray}}
\newcommand{\beqn}{\begin{eqnarray*}}
	\newcommand{\eeqn}{\end{eqnarray*}}

\begin{document}
	\title[On cylindrical symmetric projectively flat]{On cylindrical symmetric projectively flat Finsler metrics}
	\author{Newton Sol\'orzano and V\'ictor Le\'on}
	
\address{N. Sol\'orzano. ILACVN - CICN, Universidade Federal da Integração Latino-Americana, Parque tecnológico de Itaipu, Foz do Iguaçu-PR, 85867-970 - Brazil}
	\email{nmayer159@gmail.com}

\address{V. Le\'on. ILACVN - CICN, Universidade Federal da Integração Latino-Americana, Parque tecnológico de Itaipu, Foz do Iguaçu-PR, 85867-970 - Brazil}
\email{victor.leon@unila.edu.br}

	

	\begin{abstract}
We study the cylindrical symmetric Finsler metrics. We obtain the system of differential equations of such metrics which are  projectively flat. We give a family of solutions of this system. Examples are included.
	\end{abstract}
	
	\date{\today}

	\maketitle
\section{Introduction and Main Results}

A Finsler metric $ F(x,y) $ on an open subset $ U\subset \mathbb{R}^n $ is called {\em projectively flat} if all geodesics are straight lines in $ U. $ A Finsler metric $ F $ on a manifold $ M $ is called {\em locally projectively flat} if at any point, there is a locally coordinate system $ (x^i) $ where $ F $ is projectively flat. Hilbert’s Fourth Problem, which studies locally projectively flat metrics, is an important problem in Finsler geometry and projectively flat Finsler metrics on a convex domain in $ \mathbb{R}^n $ are regular solution to Hilbert’s Fourth Problem. 

G. Hamel \cite{Hamel1903} obtained conditions  which characterize projectively flat Finsler metric $ F=F(x,y) $ on an open subset $ U\subset \mathbb{R}^n. $ That is
\[F_{x^ky^l}y^k-F_{x^l}=0.\]




A Randers metric $ F=\alpha +\beta  $ is locally projectively flat if, and only if, $ \alpha $ is locally projectively flat and $ \beta $ is closed \cite{Bacso1997}. 
In \cite{HM2}, the authors founded the PDE which characterizes spherically symmetric Finsler metric and projectively flat and given an example using the Gauss error function.
In \cite{Shen2009}, the author studied $ (\alpha,\beta)$-metrics, and found a necessary and sufficient condition for such metrics to be projectively flat in dimension $ n\geq 3. $
In \cite{Liu2021}, the authors studied  locally projectively flat warped Finsler metric of constant flag curvature and obtained new examples. In \cite{Shen2016}, the authors studied the general $ (\alpha,\beta)-$metrics, and they classify those locally projectively flat when $ \alpha  $ is projectively flat. New examples were founded.

\

The following known Finsler metrics:
\begin{enumerate} 
	\item the known Shen’s fish tank metric on $M= I\times\mathbb{B}^2 \subset  \mathbb{R}^3 $ are of the form:
\begin{equation*}
F=\frac{\sqrt{(-x^2y^1+x^1y^2)^2+((y^0)^2+(y^1)^2+(y^2)^2)(1-(x^1)^2-(x^2)^2)}}{1-(x^1)^2-(x^2)^2} - \frac{x^2y^1-x^1y^2}{1-(x^1)^2-(x^2)^2}, 
\end{equation*}
where $ x=(x^0,x^1,x^2)\in \mathbb{R}\times\mathbb{B}^2  $ and $ y=(y^0,y^1,y^2)\in T_xM,$

\item the particular Randers metric constructed by Z. Shen in \cite{Shen2008}:
\[F=\frac{\sqrt{(1-\vert x\vert^4)\vert y \vert^2 + (y^0\vert x\vert^2-2x^0\langle x,y\rangle)^2} }{1-\vert x\vert^4} + \frac{y^0\vert x\vert^2-2x^0\langle x,y\rangle}{1-\vert x\vert^4},\]
where $ x=(x^0,x^1,...,x^n)\in M=I\times\mathbb{R}^n, y=(y^0,y^1,...,y^n)\in T_xM,$

\item the spherically symmetric Finsler metric \cite{HM1,Z}: \[ F=\vert y\vert \phi\left(\vert x\vert, \frac{ \langle x,y\rangle }{\vert y \vert}\right), \]
where $ x\in M=\mathbb{R}^{n+1}, y\in T_xM, $
\item the warped metrics \cite{Zhao2018,Kozma2001,Liu2019,marcal2020} defined on $ I\times \mathbb{R}^n $ of the form \begin{align*}
	F&=\vert \overline{y}\vert\phi\left(x^0,\frac{y^0}{\vert\overline{y}\vert}\right),\\
	F&=\vert \overline{y}\vert \phi\left(\frac{y^0}{\vert\overline{y}\vert},\vert \overline{x}\vert\right),
	\end{align*}
where 
 $ x=(x^0,\overline{x})=(x^0,x^1...,x^n)\in M=I\times \mathbb{R}^n, y=(y^0,\overline{y})=(y^0,y^1,...,y^n)\in T_xM,$
\end{enumerate}

are important and very studied Finsler metrics. We observed that those above metrics satisfy \begin{align}\label{cond01}	F((x^0,{O}\overline{x}),(y^0,{O}\overline{y}))=F((x^0,\overline{x}),(y^0,\overline{y})),
\end{align} for every $ O\in O(n). $ A Finsler metric $ F $ is called  {\em cylindrically symmetric} if $ F $ satisfies \eqref{cond01} for all $ O\in O(n). $ 
	
In \cite{Solorzano2022} (see Theorem \ref{teoSolor2022} in section 2) the author showed that every cylindrically symmetric Finsler metric can be written as 
\[F(x,y)=\vert \overline{y}\vert{\phi\left(x^0,\frac{y^0}{\vert \overline{y}\vert},\vert \overline{x}\vert,\frac{\langle\overline{x},\overline{y}\rangle}{\vert\overline{y}\vert}\right)},\]
where $|\cdot|$ and $\langle \cdot,\cdot\rangle$ are, respectively, the 
standard Euclidean norm and inner product on $\mathbb{R}^n$. 


In this paper, we give necessary and sufficient conditions for $ F=\vert\overline{y}\vert\phi $ to be a Finsler metric. The proof is inspired by the proof given in \cite{schernShen2005} for $ (\alpha,\beta)-$metrics.

\begin{theorem}\label{theotobeFinsler}
	Let $F=\vert\overline{y}\vert{\phi(x^0,z,r,s)}$ be a Finsler metric defined on $  M =I\times\mathbb{B}^n(\rho) $, where $ z=\frac{y^0}{\vert\overline{y}\vert}, $ $r=\vert\overline{x}\vert$, $ s=\frac{\langle\overline{x},\overline{y}\rangle}{\vert\overline{y}\vert} $ and  $TM $  with coordinates 
	\begin{equation}\label{coordxy}
		x=(x^0, \overline{x}),\;\overline{x}=(x^1,\ldots,x^n),\;y=(y^0, \overline{y})\;\mbox{ and }\overline{y}=(y^1,\ldots,y^n).
	\end{equation} 	
	 Then, $  F$ is a Finsler metric if, and only if, the positive function $ \phi $ satisfies  $ \Lambda>0 $ for $ n=2 $  with additional inequality, $ \Omega>0 $ for $ n\geq 3. $
\end{theorem}
Next, we obtain a system of differential equations which characterize cylindrical symmetric projectively flat Finsler metrics:
\begin{theorem}\label{theo:flat}
	Let $F=\vert\overline{y}\vert{\phi(x^0,z,r,s)}$ be a  Finsler metric defined on $   M =I\times\mathbb{B}^n(\rho) $, where $ z=\frac{y^0}{\vert\overline{y}\vert}, $ $r=\vert\overline{x}\vert$, $ s=\frac{\langle\overline{x},\overline{y}\rangle}{\vert\overline{y}\vert} $ and  $TM $  with coordinates   \eqref{coordxy}.  Then $ F $ is projectively flat if and only if $ \phi $ satisfy
	\begin{align}
		\Omega_{x^0}-\phi_{sz}&=0,\label{eq:flat01}\\
		\Omega_r-r\phi_{ss}&=0,\label{eq:flat02}
	\end{align}
where\[\Omega=\phi-s\phi_s-z\phi_z.\]
\end{theorem}
Studying this PDE, we obtain a family of solutions of this system:

\begin{theorem}\label{prop:family}
	The family of functions given by
	\begin{equation}\label{phiflat}
		\begin{array}{rl}\phi=&g_1(z)+x^0g_2(z)+sg_3(z)+zg_4(x^0)+sg_5(r)\\ &\\&+\displaystyle\int^s_0\left(\int^\eta_0g_6(r^2-\xi^2)d\xi\right)d\eta+\int^r_0\xi g_6(\xi^2)d\xi,\end{array} 
	\end{equation}
	where $ g_i, i=1,...,6, $ are arbitrary real differentiable functions with $ g_2(z)-zg'_2(z)-g'_3(z)=0, $  are solutions of the system \eqref{eq:flat01}-\eqref{eq:flat02}. Moreover, any Finsler metric on $   M =I\times\mathbb{B}^n(\rho) $ defined by
	\[F(x,y)=\vert\overline{y}\vert\phi\left(x^0,\frac{y^0}{\vert\overline{y}\vert},\vert\overline{x}\vert,\frac{\langle \overline{x},\overline{y}\rangle}{\vert\overline{y}\vert}\right),\]
	where the positive function $ \phi $ is given by \eqref{phiflat} with $ g_i, i=1,...,6 $ satisfy
	\begin{align}
		\Lambda&=\left(g_1(z)-zg'_1(z) + (x^0-sz)g'_3(z) + \frac{1}{2}\int_0^{\small{r^2-s^2}}g_6(\xi)\,d\xi+ (r^2-s^2)g_6(r^2-s^2)\right)\times\nonumber\\
		&\qquad\qquad\qquad\qquad\qquad\left(g''_1(z) + (x^0-sz)g''_2(z)\right) -(r^2-s^2)\left[g'_3(z)\right]^2>0,	\text{ when } n \geq 2,\label{cond1tobefinsler}
	\end{align}
	with the additional inequality
	\begin{align}\label{cond2tobefinsler}
		\Omega=	g_1(z)-zg'_1(z)+(x^0-sz)g'_3(z) + \frac{1}{2}\int_0^{r^2-s^2}g_6(\xi)\,d\xi>0,\qquad \text{ when } n\geq 3,
	\end{align}
	are locally projectively flat Finsler Metric.
\end{theorem}

Finally, using results in other works \cite{MoZoTe2013,HM1,Solorzano2022} we obtain new examples (see Section \ref{exemplesfinal}),  which are a generalization of the warped metrics mentioned above.

\section{Preliminaries}

In this section, we  give some notations, definitions and lemmas that will be used in the proof of our main results.
Let $M$ be a manifold and let $TM=\cup_{x\in M}T_xM$ be the tangent
bundle of $M$, where $T_xM$ is the tangent space at $x\in M$. We
set $TM_o:=TM\setminus\{0\}$ where $\{0\}$ stands for
$\left\{(x,\,0)|\, x\in M,\, 0\in T_xM\right\}$. A {\em Finsler
metric} on $M$ is a function $F:TM\to [0,\,\infty)$ with the
following properties:
\begin{itemize}
	\item[(a)] $F$ is $C^{\infty}$ on $TM_o$;
	\item[(b)] At each point $x\in M$, the restriction $F_x:=F|_{T_xM}$ is a
	Minkowski norm on $T_xM$.
\end{itemize}

Let  $\mathbb{B}^n(\rho)\subset\mathbb{R}^n$ the $n$ dimensional open ball of radius $ \rho $ and centered at the origin ($n\geq 2 $). Set 
 $M=I\times \mathbb{B}^n(\rho),$ with coordinates on $ TM $
\[
	x=(x^0, \overline{x}),\;\overline{x}=(x^1,\ldots,x^n),\;y=(y^0, \overline{y})\;\mbox{ and }\overline{y}=(y^1,\ldots,y^n).\]

Consider the Finsler metric $ F $ defined on $ M $ such that \begin{align}\label{eq:sph2}
	F((x^0,{O}\overline{x}),(y^0,{O}\overline{y}))=F((x^0,\overline{x}),(y^0,\overline{y}))\end{align} 
for every orthogonal $n\times n$ matrix ${O}.$

 Inspired by \cite{Z}, (see also \cite{HM2}) in \cite{Solorzano2022} was proved the following:

\begin{theorem}\cite{Solorzano2022}\label{teoSolor2022}
	A Finsler metric $ F, $ defined on $ M=\mathbb{R}\times\mathbb{B}^n(\rho), $ satisfies \eqref{eq:sph2} if, and only if, there exists a positive differentiable function $ \phi:\mathbb{R}^4\rightarrow \mathbb{R} $ such that
\[F(x,y)=\vert \overline{y}\vert{\phi\left(x^0,\frac{y^0}{\vert \overline{y}\vert},\vert \overline{x}\vert,\frac{\langle\overline{x},\overline{y}\rangle}{\vert\overline{y}\vert}\right)},\]
	where $|\cdot|$ and $\langle \cdot,\cdot\rangle$ are, respectively, the 
	standard Euclidean norm and inner product on $\mathbb{R}^n$.
\end{theorem}


For $ \vert\overline{y}\vert\neq  0,$ we introduce the notation 
\[z:= \frac{y^0}{\vert\overline{y}\vert},\;\;r:=|\overline{x}|\;\;\mbox{and}\;\;s:=\frac{\langle \overline{x},\,\overline{y}\rangle}{\vert\overline{y}\vert},\] 
where $|\,.\,|$ and $\langle \,,\rangle$ are, respectively, the 
standard Euclidean norm and inner product on $\mathbb{R}^n$. 

Throughout our work, the following convention for indices is adopted: 
\[	0\leq A, B, \ldots \leq n\;\;\mbox{and}\;\;1\leq i,j,\ldots \leq n.  
\]
%

Defining $ \Omega $ and $ \Lambda $ by
\begin{align}
	\Omega:=&\phi-s\phi_s-z\phi_z \label{DefOmega}\\
	 \Lambda:=& \Omega \phi_{zz}+(r^2-s^2)(\phi_{ss}\phi_{zz}-\phi^2_{sz}),\label{DefLambda}
\end{align}
 the matrix $ \left(g_{AB}\right)=\frac{1}{2} [F^2]_{y^Ay^B}=
\left(
\begin{array}{c|c}
	g_{00} & g_{0j} \\
	\hline
	g_{i0} & g_{ij}
\end{array}
\right),$ is given by
%
\[g_{00}=\phi^2_z+\phi\phi_{zz},\;\;
g_{i0}=g_{0i}=(\phi\Omega)_zu^i+(\phi_s\phi_z+\phi\phi_{sz})x^i,\]	
	\[\;\;\mbox{and}\;\;g_{ij}=\phi\Omega\delta_{ij} + X_{ij},\]
where $ X_{ij}=(u^i, x^i)\left(\begin{array}{c c}
	-[s(\phi\Omega)_s+z(\phi\Omega)_z] & (\phi\Omega)_s\\
	(\phi\Omega)_s& \phi_s^2+\phi\phi_{ss} 
\end{array}\right)\left(\begin{array}{c}
	u^j\\
	x^j
\end{array}\right),$ with $ u^i=\frac{y^i}{\vert\overline{y}\vert}. $  
 
Note that,
\[\det(g_{AB})={(\phi_z^2+\phi\phi_{zz})^{1-n}}\phi^n\det\left((\phi_z^2+\phi\phi_{zz})\Omega\delta_{ij}+\delta_0u^iu^j + \delta_1(x^i + u^i)(x^j+u^j) + \delta_2x^ix^j
\right),\]
where,
\begin{align*}
	\delta_0=&-\left[(\phi\Omega)_z\left((\phi\Omega)_z-\phi_s\phi_z-\phi\phi_{sz}\right) + (\phi_z^2+\phi\phi_{zz})\left((s+1)(\phi\Omega)_s+z(\phi\Omega)_z\right)\right], \\
	\delta_1=&(\phi\Omega)_s(\phi_z^2+\phi\phi_{zz})-(\phi\Omega)_z(\phi_s\phi_z+\phi\phi_{sz}),\\
	\delta_2=&(\phi\phi_{zz}+\phi_z^2)\left(\phi\phi_s+\phi_s^2-(\phi\Omega)_s\right)-(\phi_s\phi_z+\phi\phi_{sz})\left(\phi_s\phi_z+\phi_{sz}^2 - (\phi\Omega)_z\right).
\end{align*}
Then, the determinant of $ g_{AB} $ is given by \begin{align}\label{detgAB}
	\det(g_{AB})=\phi^{n+2}\Omega^{n-2}\Lambda.
\end{align}

As mentioned in \cite{Kozma2001}, the called warped metrics $ F=\sqrt{F_1^2 + f^2F_2^2}, $ where $ (M,F_1), (N,F_2) $ are Finsler manifolds and $ f $ is a smooth function on $ M, $ are not necessarily smooth on the vectors of the form $ (v_1,0) $ and $ (0,v_2) \in TM\times TN.$ However, the first two Shen's examples in the introduction part motivate us to formalize a Finsler metric which satisfies \eqref{eq:sph2} although warped type metrics can be obtained by similar technician.

\section{Proof of Theorems \ref{theotobeFinsler}, \ref{theo:flat} and \ref{prop:family}}

\begin{proof}[Proof of Theorem \ref{theotobeFinsler}]
Suppose that $ \Omega, \Lambda >0. $ From $ \Lambda>0 $ we conclude that $ \phi_{zz}>0. $ In fact, suppose that there exist $ x^0_0,z_0,r_0,s_0$  such that $ \phi_{zz}({x^0_0,z_0,r_0,s_0})=0 $, then $ \Lambda|_{x^0_0,z_0,r_0,s_0}=-(r^2-s^2)\phi_{sz}^2\leq 0 $ and, if $ \phi_{zz}<0, $ hence $ \Lambda|_{r=s=0} = \Omega\phi_{zz}\vert_{r=s=0}<0.$ Additionally, from $ \Lambda>0 $ we obtain $ \Omega+(r^2-s^2)\phi_{ss}> 0. $

 Now, we consider the family of functions $ \phi_t $ defined by,
	\[\phi_t:=(1-t)\sqrt{1+z^2}+t\phi.\]
Let $ F_t:=\vert \overline{y}\vert\phi_t $ and $ g_{AB}^t:=\frac{1}{2}\left[F_t^2\right]_{y^Ay^B}. $ Note that, for any $ 0\leq t\leq 1, $ we have
\begin{align*}
	\Omega_t:=&\phi_t-s(\phi_t)_s-z(\phi_t)_z=\frac{1-t}{\sqrt{1+z^2}}+t\Omega>0,\\
	\Lambda_t:=&\Omega_t(\phi_t)_ {zz}+(r^2-s^2)\left((\phi_t)_{ss}(\phi_t)_{zz}-(\phi_t)^2_{sz}\right)\\
	=&\frac{(1-t)^2}{(1+z^2)^2}+\frac{t(1-t)}{\sqrt{1+z^2}}\phi_{zz}+\frac{t(1-t)}{(1+z^2)^{3/2}}\left(\Omega+(r^2-s^2)\phi_{ss}\right)+t^2\Lambda>0.
\end{align*}
Thus $ \det(g^t_{AB})>0 $ for all $ 0\leq t\leq 1. $ Since $ (g^0_{AB}) $ is positive definite, we conclude that $ (g^t_{AB}) $ is positive definite for any $ t\in [0,1]. $	Therefore, $ F_t $ is a Finsler metric for any $ t\in [0,1]. $

Conversely, assume that $ F=\vert\overline{y}\vert\phi(x^0,z,r,s) $ is a Finsler metric on $ M. $ Then $ \phi>0. $ By \eqref{detgAB}, $ \det(g_{AB})>0 $ is equivalent to
\begin{align}\label{desigOL}
	\Omega^{n-2}\Lambda>0,
\end{align} 
which implies $ \Omega\neq 0 $ when $ n\geq 3. $ Since $ \phi\vert_{s=z=0}>0, $ we have $ \Omega\vert_{s=z=0}>0 $, then $ \Omega>0. $ And by \eqref{desigOL}  we have $ \Lambda>0. $

\end{proof}


The inverse of $(g_{AB})$ is given by,
\[ 
\left(g^{AB}\right)=
\frac{1}{\phi^4\Lambda}\left(
\begin{array}{c|c}
	y^{00} & y^{0j} \\
	\hline
	y^{i0} & \frac{\phi^3\Lambda}{\Omega}\delta_{ij}+ Y_{ij} 
\end{array}
\right)
\]
where
\begin{align*}
	y^{00}=& \phi\Omega((\phi-z\phi_z)^2+z^2\phi\phi_{zz}) + (r^2-s^2)\phi \left(\phi^2\phi_{ss}+2z\phi(\phi_s\phi_{sz}-\phi_z\phi_{ss}) + z^2\delta_3\right)\\
	y^{0i}=&\phi\left[-(\Omega+s\phi_s)(\phi\Omega)_z+(r^2-s^2)(\phi(\phi_s\phi_{sz}-\phi_z\phi_{ss})+z\delta_3)\right]u^i + \phi^2[\phi_s\Omega_z-\phi_{sz}\Omega]x^i
\end{align*}
with
\begin{align*}
	\delta_3=&\phi(\phi_{ss}\phi_{zz}-\phi_{sz}^2)+\phi_s(\phi_s\phi_{zz}-\phi_z\phi_{sz})+\phi_z(\phi_{ss}\phi_z-\phi_s\phi_{sz})\\
	=&-\det \left(
	\begin{array}{c c c}
		-\phi & \phi_s&\phi_z \\
		\phi_s & \phi_{ss}&\phi_{sz}\\
		\phi_z&\phi_{sz}&\phi_{zz}
	\end{array}
	\right),
\end{align*}
and 
\begin{align*}
	Y_{ij}=&	\frac{1}{\phi\Omega}(u^i, x^i)\left(\begin{array}{c c}
		y_{11} & y_{12}\\
		y_{12}& y_{22} 
	\end{array}\right)\left(\begin{array}{c}
		u^j\\
		x^j
	\end{array}\right) 
\end{align*}
where
\begin{align*}
	y_{11}=&\phi^2\left[(\phi\Omega)_z^2+\phi\phi_{zz}(z(\phi\Omega)_z+s(\phi\Omega)_s)-(r^2-s^2)(\phi^2(\phi_{ss}\phi_{zz}-\phi_{sz}^2)-\Omega\delta_2)\right]\\
	y_{12}=&y_{21}=\phi^3\left[\phi_{sz}(\phi\Omega)_z-\phi_{zz}(\phi\Omega)_s\right]\\
	y_{22}=&-\phi^4(\phi_{ss}\phi_{zz}-\phi_{sz}^2).
\end{align*}
The next proposition gives us one the most important quantities in Finsler Geometry: The geodesic coefficients \begin{align}\label{def:GA}
	 G^A=Py^A+Q^A, \end{align}
where
\begin{equation}\label{eq:PQ}
	P:=\frac{F_{x^C}y^C}{2F}\;\;\;\mbox{and}\;\;\;Q^A:=\frac{F}{2}g^{AB}\left\{F_{x^Cy^B}y^C-F_{x^B}\right\},
\end{equation}
for a cylindrically simmetric Finsler metric $ F.$

\begin{proposition}
	$F=\vert\overline{y}\vert{\phi(x^0,z,r,s)}$ be a Finsler metric defined on $ M =I\times\mathbb{B}^n(\rho)$, where $ z=\frac{y^0}{\vert\overline{y}\vert}, $ $r=\vert\overline{x}\vert$, $ s=\frac{\langle\overline{x},\overline{y}\rangle}{\vert\overline{y}\vert} $ and  $TM $  with coordinates   \eqref{coordxy}.  Then the geodesic spray coefficients $ G^A $ are given by
	\begin{align*}
		G^0=&u^2\left\{z(W+sU)+\frac{\Omega}{2\Lambda}(\varphi_z-2\phi_{x^0})-(r^2-s^2)V\right\},\\
		G^i=&uWy^i + u^2Ux^i,
	\end{align*}
where  $ \Omega, \Lambda $ are  in \eqref{DefOmega}, \eqref{DefLambda} respectively, $ u=\vert\overline{y} \vert$, and
\begin{align*}
	\varphi &:=z\phi_{x^0}+\frac{s}{r}\phi_r+\phi_s,\\
	W&:=\frac{1}{\phi}\left\{\frac{\varphi}{2}-s\phi U - \frac{\phi_z\Omega}{2\Lambda}(\varphi_z-2\phi_{x^0})-(r^2-s^2)\left[{\phi_s}U-{\phi_z}V\right]\right\},\\
	U&:=\frac{1}{2\Lambda}\left\{\left(\varphi_s-\frac{2}{r}\phi_r\right)\phi_{zz}-\left(\varphi_z-2\phi_{x^0}\right)\phi_{sz}\right\},
\end{align*}
\begin{equation}\label{def:V}
V:=\frac{1}{2\Lambda}\left\{\left(\varphi_s-\frac{2}{r}\phi_r\right)\phi_{sz}-\left(\varphi_z-2\phi_{x^0}\right)\phi_{ss}\right\}.
\end{equation}

\end{proposition}

\begin{proof}

The derivatives of $ F=\vert\overline{y}\vert\phi, $ with respect to $ x^0, x^i, y^0 $ and $ y^i $ are given by
	\begin{align*}
		F_{x^0}&=u\phi_{x^0},\\
		F_{x^i}&=u\left(\phi_su^i+ {\phi_r}\frac{x^i}{r}\right),\\
		F_{y^0}&=\phi_z,\\
		F_{x^0y^0}&= \phi_{x^0z},\\
		F_{x^iy^0}&=\phi_{sz}u^i+\phi_{rz}\frac{x^i}{r},\\
		F_{y^i}&=\Omega u^i+\phi_s {x^i},\\
		F_{x^0y^i}&=\Omega_{x^0}u^i + \phi_{x^0s}x^i,\\
		F_{x^iy^j}&=\phi_s\delta_{ij}+ \Omega_s u^iu^j + \phi_{ss}\frac{y^i}{u}x^j+\Omega_r\frac{x^i}{r}u^j + \frac{1}{r}\phi_{rs}x^ix^j.
	\end{align*}
	Then,
	\begin{align*}
		F_{x^Cy^0}y^C-F_{x^0}&=u\left(\varphi_z-2\phi_{x^0}\right),\\
		F_{x^Cy^j}y^C-F_{x^j}&=u\left(\left[-s(\varphi_s-\frac{2}{r}\phi_r)-z(\varphi_z-2\phi_{x^0})\right]u^j+\left[\varphi_s-\frac{2}{r}\phi_r\right]x^j\right).
	\end{align*}
	Thus, from \eqref{eq:PQ}, we have
	\begin{align}
		P=&\frac{u^2}{2F}\varphi\label{PU},\\
		Q^0=&u^2\left\{\frac{1}{2\phi\Lambda}(\phi-z\phi_z)(\varphi_z-2\phi_{x^0})\Omega - (r^2-s^2)\left(z\frac{\phi_s}{\phi}U + \frac{\phi-z\phi_z}{\phi}V\right)\right\}\label{Q0U},\\
		Q^i=&-u^2\left\{sU+(r^2-s^2)(\frac{\phi_s}{\phi}U-\frac{\phi_z}{\phi}V) + \frac{\phi_z\Omega}{2\phi\Lambda}(\varphi_z-2\phi_{x^0})\right\}\frac{y^i}{u}+{u^2}Ux^i\label{QiU}.
	\end{align}
Replacing \eqref{PU}, \eqref{Q0U} and \eqref{QiU} into \eqref{def:GA}, we have the result.
\end{proof}

A Finsler metric $  F $ is \textit{projectively flat} if, and only if, $ F $ satisfies \[F_{x^ky^l}y^k-F_{x^l}=0.\]
From \eqref{Q0U} and \eqref{QiU} we obtain a characterization of $ F=\vert \overline{y}\vert\phi $ which are projectively flat. 
\begin{proof}[Proof of Theorem \ref{theo:flat}]
	From \eqref{Q0U} and \eqref{QiU}, the Finsler metric  $F=\vert\overline{y}\vert{\phi(x^0,z,r,s)} $ is projectively flat if, and only if, 
	\begin{equation}\label{systemflat01}
		\left\{\begin{array}{rl}
			U&=0,\\
		(r^2-s^2)\frac{\phi_z}{\phi}V&=\frac{\phi_z}{2\phi\Lambda}\Omega(\varphi_z-2\phi_{x^0}),\\
	(r^2-s^2)\left(\frac{\phi-z\phi_z}{\phi}\right)V&=\frac{\Omega}{2\phi\Lambda}(\phi-z\phi_z)(\varphi_z-2\phi_{x^0}).
	\end{array}\right.
\end{equation}
	Using the fact $ \phi-z\phi_z \neq 0, $   the  system \eqref{systemflat01} is equivalent to  
	\begin{align}\label{eq:flat0}
		\left\{\begin{array}{rl}
			(\varphi_s-\frac{2}{r}\phi_r)\phi_{zz}&=(\varphi_z-2\phi_{x^0})\phi_{sz},\\
		(\varphi_z-2\phi_{x^0})\Omega&=2(r^2-s^2)\Lambda V.
		\end{array}\right.
	\end{align}
From the first equation of \eqref{eq:flat0} and \eqref{def:V} we have
 \[V=\frac{1}{2\phi_{zz}\Lambda}(\varphi_z-2\phi_{x^0})(\phi^2_{sz}-\phi_{ss}\phi_{zz}).\] 
 Substituting this in the second equation of \eqref{eq:flat0}, and using the fact $ \Lambda\neq 0, \forall y$ we have that the system \eqref{eq:flat0} implies
\[\varphi_z-2\phi_{x^0}=0,\] substituting this into the system \eqref{eq:flat0}, we obtain that the system \eqref{eq:flat0}	is equivalent to \begin{align}
	r(\phi_{x^0}-z\phi_{x^0z}-\phi_{sz})-s\phi_{rz}&=0,\label{eq:flat1}\\
	s\phi_{rs}+r(\phi_{ss}+z\phi_{x^0s}) - \phi_r&=0.\label{eq:flat2}
\end{align} Derivating \eqref{eq:flat1} and \eqref{eq:flat2} in relation to $ s $ and $ z $ respectively, and subtracting them, we have 
\begin{align}\label{eq:resolv}
\phi_{rz}-r\phi_{x^0s}=0,
\end{align} 
using \eqref{eq:resolv} in  equations \eqref{eq:flat1} and \eqref{eq:flat2} we have the result.
The reciprocal is immediately.
\end{proof}

\begin{proof}[Proof of Theorem \ref{prop:family}]
	To construct this family, note that $\phi$ given by  
	\begin{align}\label{eq:phifamily}
		\phi(x^0,z,r,s)=\varepsilon_1(x^0,r)+\varepsilon_2(x^0,z)+\varepsilon_3(r,s)+\varepsilon_4(s,z),
	\end{align}
	where $\varepsilon_1,\varepsilon_2, \varepsilon_3$ and $\varepsilon_4$ are real differentiable functions, satisfy  \eqref{eq:resolv}.
	
	Observe that equations \eqref{eq:flat01} and \eqref{eq:flat02}, in this case, become
	\begin{align}
		\Omega_{x^0}-\phi_{sz}&=[\varepsilon_1]_{x^0}+[\varepsilon_2]_{x^0}-z[\varepsilon_2]_{x^0z}-[\varepsilon_4]_{sz}=0\label{eq:omegax},\\
		\Omega_r-r\phi_{ss}&=[\varepsilon_1]_r+[\varepsilon_3]_r-s[\varepsilon_3]_{rs}-r[\varepsilon_3+\varepsilon_4]_{ss}=0.\label{eq:omegar}
	\end{align} 
	Taking the derivative of \eqref{eq:omegax} with respect to $r$, we have 
	\[[\varepsilon_1]_{x^0r}=0,\]
	therefore exists real differential functions $f_1,f_2:\mathbb{R}\rightarrow\mathbb{R} $ such that
	\begin{align}\label{epsilon1}
		\varepsilon_1=f_1(x^0)+f_2(r).
		\end{align}
	Now, taking the derivative of \eqref{eq:omegax} with respect to $x^0$ we get
	\[[\varepsilon_2]_{x^0x^0}-z[\varepsilon_2]_{x^0x^0z}= -f''_1(x^0),\]
	therefore, there exists real differentiable functions $f_3,f_4,f_5:\mathbb{R}\rightarrow\mathbb{R} $ such that
	\begin{align}\label{epsilon2}
		\varepsilon_2=-f_1(x^0)+zf_3(x^0)+x^0f_4(z) + f_5(z).
	\end{align}
	Again, taking the derivative of \eqref{eq:omegax} with respect to $s$, we have 
	\[[\varepsilon_4]_{zss}=0,\]
	therefore there exists real differential functions $f_6,f_7,f_8:\mathbb{R}\rightarrow\mathbb{R} $ such that
	\begin{align}\label{epsilon4}
		\varepsilon_4=f_6(s) + f_7(z) + sf_8(z).
		\end{align}
	Also, taking the derivative of \eqref{eq:omegar} with respect to $s$, we have
	\[s\Big[[\varepsilon_3]_{ss}\Big]_{r}+r\Big[[\varepsilon_3]_{ss}\Big]_{s}=-rf^{\prime\prime\prime}_6(s),
	\]
	now using characteristic equations in this non-homogeneous quasi linear PDE, we obtain that there exists a real differentiable function $f_9:\mathbb{R}\rightarrow\mathbb{R}$ such that
	\[[\varepsilon_3]_{ss}=f_9(r^2-s^2)-f^{\prime\prime}_6(s),\]
	consequently, together \eqref{eq:omegar}, there exists real differential function $f_{10}:\mathbb{R}\rightarrow\mathbb{R}$ and constant $k\in\mathbb{R}$ such that
	\begin{align}\label{epsilon3}
		\varepsilon_3=k-f_2(r)-f_6(s)+sf_{10}(r)+\int^s_0\left(\int^\eta_0f_9(r^2-\xi^2)d\xi\right)\,d\eta+\int^r_0 \rho f_9(\rho^2)\,d\rho.
	\end{align}
	Also substituting \eqref{epsilon1}, \eqref{epsilon2}, \eqref{epsilon4} and $ \eqref{epsilon3}  $ in \eqref{eq:phifamily} and using  \eqref{eq:omegax}, we have
	\[f_4(z)-zf'_4(z)-f'_8(z)=0.\]
	Finally, the identity \begin{align}\label{identity}
		\int^s_0\int^\eta_0g_6(r^2-\xi^2)d\xi d\eta+\int^r_0 \rho g_6(\rho^2)d\rho=&\frac{1}{2}\int_0^{r^2-s^2}g_6(\xi)\,d\xi + s\int_0^sg_6(r^2-\xi^2)\,d\xi,
	\end{align} and the Theorem \ref{theotobeFinsler} give us the conditions \eqref{cond1tobefinsler}-\eqref{cond2tobefinsler}.
\end{proof}
 
\begin{remark} As seen in the proof above, we give another useful expression for the general solution of \[s\phi_{bs}+b\phi_{ss}-\phi_b=0\] founded in Proposition 5.1 in \cite{HM2}:
	\begin{align}\label{Mosol}
		\phi(b,s)=sh(b)-s\int_{s_0}^st^{-2}g(b^2-t^2)\,dt.
	\end{align}
	In fact, it is enough to take the derivative of the equation with respect to $s$, to getting 
	\[s\Big[\phi_{ss}\Big]_{b}+b\Big[\phi_{ss}\Big]_{s}=0.
	\]
	Using characteristic equations in the homogeneous quasi linear PDE, we obtain that there exists a real differentiable function $f:\mathbb{R}\rightarrow\mathbb{R}$ such that
	\[\phi_{ss}=f(b^2-s^2),\]
	therefore, there exists real differentiable functions $g,h:\mathbb{R}\rightarrow\mathbb{R}$ and constant $k\in\mathbb{R}$ such that
	\begin{align}\label{useful}
		\phi=k+sg(b)+\int^s_0\left(\int^\eta_0f(b^2-\xi^2)d\xi\right)d\eta+\int^b_0 \rho f(\rho^2)d\rho,
	\end{align}
by the identity \eqref{identity}
we have that conditions $ \phi-s\phi_s>0 $ and $ \phi-s\phi_s+(b^2-s^2)\phi_{ss}>0 $ become
\begin{align}\label{conditionspherica1} k+\dfrac{1}{2}\displaystyle \int^{b^2-s^2}_0f(u)du>0,\;\forall\;b\geq|s|, 
	\end{align} and \begin{align}\label{conditionspherica2}
	k+\dfrac{1}{2}\displaystyle \int^{b^2-s^2}_0f(\xi)d\xi+(b^2-s^2)f(b^2-s^2)>0,\;\forall\;b\geq|s|, 
	\end{align} respectively.

 The relation between $ g $ and $ f $ from \eqref{Mosol} and \eqref{useful} is given by:
\[f=2g'.\]
\begin{remark}
Observe that, for every constant $ k>0 $ and any function $ f $ such that $ f(\xi)\geq 0, \forall \xi\geq 0,$ the conditions \eqref{conditionspherica1}-\eqref{conditionspherica2} are satisfied.
\end{remark}

\end{remark}

\begin{remark}
	The identity \eqref{identity}  give us two forms to write the function $ \phi $ in \eqref{phiflat}. 
\end{remark}

We can use functions founded in \cite{MoZoTe2013}, \cite{HM1} and \cite{Solorzano2022} for construct  warped-type Finsler metrics, i.e.
\begin{corollary}\label{corollary01}
	Let $ \phi=\phi(x^0,z,r,s) $ be a positive function defined by
	\begin{align}\label{Cor:01}
		\phi=k+g_1(z)+zg_4(x^0)+sg_5(r)+\frac{1}{2}\int_0^{r^2-s^2}g_6(\xi)\,d\xi + s\int_0^sg_6(r^2-\xi^2)\,d\xi,
	\end{align}
where $ k\in\mathbb{R} $ is  constant and $ g_1,g_4,g_5, g_6:\mathbb{R}\rightarrow\mathbb{R}, $ are differentiable functions, such that
	\begin{enumerate}
		\item[(a)] $g_1(z)+zg_4(x^0)>0,\;g_1(z)-zg_1^\prime(z)>0\;\mbox{and}\;g_1^{\prime\prime}(z)>0,\;\forall z\in\mathbb{R}$;
		\item[(b)]  $k+ \dfrac{1}{2}\displaystyle \int^{r^2-s^2}_0g_6(\xi)d\xi+(r^2-s^2)g_6(r^2-s^2)\geq0,\;\forall\;r\geq|s|,  $ when $ n\geq 2, $\\
		with additional inequality $k+\dfrac{1}{2}\displaystyle \int^{r^2-s^2}_0g_6(u)du\geq 0,\;\forall\;r\geq|s|,$ when $ n\geq 3, $
	\end{enumerate} 
then, the  warped type Finsler metric on $ M =I\times \mathbb{B}^n$ given by \[ F=\vert\overline{y}\vert\phi\left(\frac{y^0}{\vert\overline{y}\vert},\vert\overline{x}\vert,\frac{\langle \overline{x},\overline{y}\rangle}{\vert\overline{y}\vert}\right), \]  is locally projectively flat.
\end{corollary}
\begin{proof}		In Theorem \ref{prop:family} we considered $g_2(z)=g_3(z)=0$ (the result is similar if we consider $ g_3(z)=0 $ uniquely).   By  $ (a) $,  $\phi_{zz}=g_1^{\prime\prime}(z)>0,\;\forall z\in\mathbb{R}.$ Consequently, from $ (a), (b), $ \eqref{cond1tobefinsler} and \eqref{cond2tobefinsler},
	\begin{align*}
		\Omega=&k+g_1(z)-zg_1^\prime(z)+\dfrac{1}{2}\displaystyle \int^{r^2-s^2}_0g_6(\xi)d\xi>0,\;\forall z\in\mathbb{R},\;r\geq|s|;\\
		 \Lambda=&\Omega \phi_{zz}+(r^2-s^2)g_6(r^2-s^2)\phi_{zz}>0,\;\forall z\in\mathbb{R},\;r\geq|s|.
	\end{align*}
\end{proof}
Observe that if $ k>0 $ and  $ g_6(\xi)\geq 0, \forall \xi\geq 0$ then, the condition $ (b), $ in Corollary \ref{corollary01}, is satisfied.

\section{Examples}\label{exemplesfinal}
\begin{example}
	In Corollary \ref{corollary01}, by choosing
	 $g_1(z)=\sqrt{z^2+\varepsilon} + \gamma z$ ($\varepsilon>0, \vert \gamma\vert <1$) (see Example 2 in \cite{Solorzano2022}), $g_6(\xi)=\frac{2-\mu(1+(1+\mu)\xi)}{(1+\mu\xi)^{5/2}}$ (see Corollary 6.2 in \cite{MoZoTe2013}) and $ g_4(z)=0 $, then, the following  warped type Finsler metric
	 \begin{align*}
	 	F(x,y)= {{\sqrt{(y^0)^2+ \varepsilon\vert\overline{y}\vert^2} }} + \gamma y^0 + sg_5(\vert\overline{x}\vert) + \frac{\left[1+(1+\mu)\vert\overline{x}\vert^2\right]\left[\vert\overline{y}\vert^2\vert\overline{x}\vert^2-\langle\overline{x},\overline{y}\rangle^2\right] + \langle\overline{x},\overline{y}\rangle^2}{(1+\mu\vert\overline{x}\vert^2)\sqrt{\vert\overline{y}\vert^2 +\mu(\vert\overline{x}\vert^2\vert\overline{y}\vert^2-\langle\overline{x},\overline{y}\rangle)}},
	 \end{align*} 
	 is locally projectively flat. Here, $ g_5(\vert \overline{x}\vert) $ is such that $ \phi>0, $ for example $ g_5(r)=\frac{2\sqrt{1+(1+\mu)r^2}}{(1+\mu r^2)^2}. $ 
\end{example}
\begin{example}
	In Corollary \ref{corollary01}, by choosing
$g_1(z)=\sqrt{z^2+\varepsilon} $ ($\varepsilon>0$), $g_6(\xi)=2\xi^{m}$, $ m=0,1,... $ (see similar form in Theorem 1.1 in \cite{HM1}) and $ g_4(z)=0 $, we have that
\begin{align}
	\phi=k+\sqrt{z^2+\varepsilon} + sg_5(r) + \frac{(r^2-s^2)^{m+1}}{m+1}+\frac{2s}{1+2m}I_m(r,s),
\end{align}
where $ I_m=x(r^2-s^2)^m+2mr^2I_{m-1}, $ with  $ I_0(r,s)=s. $
 Then, the following warped type Finsler metric on $ \mathbb{R}\times\mathbb{B}^n, $
 \[F(x,y):=\vert\overline{y}\vert \phi \left(\frac{y^0}{\vert\overline{y}\vert},\vert\overline{x}\vert,\frac{\langle\overline{x},\overline{y}\rangle}{\vert\overline{y}\vert}\right)\]
 is locally projectively flat.
\end{example}


%

\subsection*{Acknowledgements} The first author would like to thank Prof. Keti Tenenblat for helpful conversations during her postdoctoral project at the Universidade de Brasilia (Brazil). Her suggestions were invaluable.


\end{document}